\documentclass{amsart}
\usepackage{main}

 \title[Breakdown of analyticity in Hartree scattering] 
 {Analyticity and infinite breakdown of regularity in mass-subcritical Hartree scattering}

\author[Gyu Eun Lee]{}


\begin{document}

\maketitle

\centerline{\scshape Gyu Eun Lee}
\medskip
{\footnotesize
 \centerline{Department of Mathematics, University of California, Los Angeles}
   \centerline{ Los Angeles, CA 90095, USA}
} 

\bigskip

\begin{abstract}
  We study the asymptotic behavior of solutions to the defocusing mass-subcritical Hartree NLS $iu_t + \Delta u = F(u) = (|x|^{-\gamma}*|u|^2)u$ on $\mathbb{R}^d$, $d\geq 2$, $\frac{4}{3} < \gamma < 2$.
  We show that the scattering problem associated to this equation is analytically well-posed in the weighted spaces $\Sigma = H^1\cap\mcal{F}H^1$ and $\mcal{F}H^1$.
  Furthermore, we show that the same problem fails to be analytically well-posed for data in $L^2$.
  This constitutes an infinite loss of regularity between the scattering problems in weighted spaces and in $L^2$.
  This further develops an earlier investigation initiated by the author in which a finite breakdown of regularity was proved for the $L^2$ scattering problem for the mass-subcritical NLS with power nonlinearity $F(u) = |u|^pu$.
\end{abstract}


\section{Introduction}\label{sec:intro}

Our primary object of interest in this paper is the \ita{defocusing mass-subcritical Hartree equation} (HNLS)
\begin{equation}\label{eqn:HNLS}
	iu_t + \Delta u = (|x|^{-\gamma}*|u|^2)u, ~u = u(t,x)\in\bb{C}, ~(t,x)\in\bb{R}\times\bb{R}^d,
\end{equation}
where $0 < \gamma < 2$, $d\geq 1$.
This equation, alongside its close cousin the \ita{defocusing mass-subcritical nonlinear Schr\"odinger equation} (pNLS)
\begin{equation}\label{eqn:pNLS}
	iu_t + \Delta u = |u|^pu, ~0 < p < \frac{4}{d},
\end{equation}
have been subjects of intense study from both physical and mathematical perspectives over the past few decades.
The term \ita{defocusing} refers to the positive sign of the potential $|x|^{-\gamma}*|u|^2$ for HNLS, $|u|^p$ for pNLS.
The term \ita{mass-subcritical} refers to the condition $0 < \gamma < 2$ for HNLS and $0<p<\frac{4}{d}$ for pNLS, under which the scaling symmetry of the equations is subcritical with respect to the $L^2$ norm.
It is well-known that both equations obey conservation of mass
\begin{equation}\label{eqn:mass_conservation}
	M(u(t)) = \int_{\bb{R}^d} |u(x)|^2~dx = M(u(0))
\end{equation}
and conservation of energy
\begin{equation}\label{eqn:energy_conservation}
	E(u(t)) = \frac{1}{2}\int_{\bb{R}^d}|\nabla u(x)|^2~dx + Q(u(t)) = E(u(0)),
\end{equation}
where the potential energy term $Q(u(t))$ is given by
\begin{equation}\label{eqn:HNLS_potential_energy}
	Q(u(t)) = \frac{1}{4}\int_{\bb{R}^d} (|x|^{-\gamma}*|u|^2)|u(x)|^2~dx	
\end{equation}
in the case of HNLS and by
\begin{equation}\label{eqn:pNLS_potential_energy}
	Q(u(t)) = \frac{1}{p+2} \int_{\bb{R}^d} |u(x)|^{p+2}~dx	
\end{equation}
in the case of pNLS.

In this work, we study aspects of the asymptotic behavior of solutions to HNLS in the spaces $L^2(\bb{R}^d)$, $\mcal{F}H^1(\bb{R}^d) = L^2(\bb{R}^d;|x|^2dx)$, and $\Sigma = H^1(\bb{R}^d)\cap\mcal{F}H^1(\bb{R}^d)$.
Here $\mcal{F}$ denotes the Fourier transform.
The typical conjecture for the asymptotic behavior of solutions to defocusing partial differential equations of dispersive type like HNLS and pNLS is \ita{scattering}, i.e. convergence to a free (linear) evolution.
This conjecture consists of two sub-questions, which are \ita{asymptotic completeness} and the \ita{existence of the wave operators}.

For $X$ a Banach space, we say that asymptotic completeness holds (in forward time) for HNLS(resp. pNLS) if for all initial data $u_0\in X$, there exists $u_+\in X$ such that the global solution $u\in C_{t,\tnm{loc}}X$ to HNLS (resp. pNLS) with Cauchy data $u(t=0) = u_0$ converges to a free evolution in the sense that
\[
	\lim_{t\to +\infty}\|e^{-it\Delta}u(t) - u_+\|_X = 0.
\]
When this occurs, we say that $u(t)$ \ita{scatters} to $u_+$.
In this case we may define the \ita{initial-to-scattering-state operator}
\begin{equation}
	\mcal{S}:X\to X: \mcal{S}(u_0) = u_+.
\end{equation}
Similarly, we say that the (forward) wave operator exists on $X$ for HNLS (resp. pNLS) if for all final states $u_+\in X$, there exists a unique global solution $u\in C_{t,\tnm{loc}}X$ to HNLS (resp. pNLS) which scatters to $u_+$.
When this holds, we may define the \ita{wave operator}
\begin{equation}
	\mcal{W}:X\to X:\mcal{W}(u_+) = u(t=0).
\end{equation}
When the two operators exist, they are necessarily inverses of each other.
Analogous definitions can be made in backwards time, i.e. as $t\to-\infty$.
When asymptotic completeness holds and the wave operator exists on $X$ (in both forward and backward time), we say that \ita{scattering} holds for the equation on $X$.

For both HNLS and pNLS, there is a considerable disparity in our understanding of the asymptotic behavior of solutions having initial/final states in $L^2$ and of solutions with initial/final states in weighted spaces.
The vast majority of the scattering theory for these equations is built under the assumption that the data lies in $\Sigma, \mcal{F}H^1$, or a similar weighted space.
In contrast, there is nearly no literature on the behavior of solutions for either equation under the weaker assumption of $L^2$ initial/final data, in spite of the fact that the equations are globally well-posed in $L^2$ and obey conservation of $L^2$ norm.
Therefore we are in a curious situation where we know that all solutions to HNLS and pNLS with $L^2$ initial data are global, but we have no understanding of the asymptotic behavior of a vast majority of such solutions.

In this paper, we continue investigations initiated by the author in \cites{Lee21}, in which we sought to give an explanation for this disparity in the case of pNLS.
Our main result was as follows:
\begin{thm}[\cites{Lee21}, Theorem 1.3]\label{thm:pNLS_main}
	Let $d\geq 1$, and consider pNLS with $\alpha(d) < p < \frac{4}{d}$, where
	\[
		\alpha(d) = \frac{2-d + \sqrt{(d-2)^2 + 16d)}}{2d}	
	\]
	denotes the Strauss exponent.
	Then:
	\begin{enumerate}
		\item The scattering operators $\mcal{S},\mcal{W}$ for pNLS are well-defined as maps $\Sigma\to L^2$, and are maximally regular in the sense that they are H\"older continuous of order $1+p$ at $0\in \Sigma$, but not of any higher order.
		\item There exists $\beta = \beta(d,p)\in (0,p)$ such that $\mcal{S},\mcal{W}$ admit no extensions to maps $L^2\to L^2$ which are H\"older continuous of order $1+\beta$ at $0\in L^2$.
	\end{enumerate}
\end{thm}
Here, we define the notion of H\"older continuity at a point through membership in the pointwise H\"older spaces $C^s(x_0)$, which were introduced in \cites{Andersson97}.
Membership in $C^s(x_0)$ is a necessary condition for the existence of $n$ Gateaux derivatives at $x_0$ and $r$-H\"older continuity of the $n$-th Gateaux derivative, where $n$ is the integer part of $s$ and $r$ its fractional part.
We refer the reader to Definition 2.2 and the appendix of \cites{Lee21} for details.

We interpret this result as a well-posedness result on the scattering problem for pNLS for initial/final data in $\Sigma$, and as an ill-posedness result for the scattering problem for data in $L^2$ in the sense of Bourgain \cites{Bo97}.
It says that any hypothetical extension of the scattering theory of pNLS from $\Sigma$ to $L^2$ must come at a cost.
The corresponding extensions of the scattering operators lose a positive amount of regularity through the extension, and in particular they fail to have the expected regularity $C^{1+p}$ that one would expect from the smoothness of the pNLS nonlinearity $F(u) = |u|^pu$.

Our goal in this work is to extend and expand upon these results in the case of the Hartree equation \eqref{eqn:HNLS}.
We now state the main results of this paper:

\begin{thm}[Analyticity of the Hartree scattering operators]\label{thm:main1}
	Let $d\geq 2$ and $\frac{4}{3} < \gamma < 1$.
	Let $\mcal{T}\in\{\mcal{S},\mcal{W}\}$.
	Then:
	\begin{enumerate}
		\item $\mcal{T}$ is well-defined as a map $\Sigma\to \Sigma$, and is analytic in the sense that for all $u_0\in \Sigma$ and $v\in\Sigma$, $\mcal{T}$ admits the power series expansion
		\[
			\mcal{T}(u_0 + \ve v) = \mcal{T}(u) + \sum_{k=1}^\infty \ve^kw_k
		\]
		for all sufficiently small $\ve> 0$, where $(w_k)\subset \Sigma$ and the series converges in $\Sigma$-norm.
		\item The same result holds with the space $\mcal{F}H^1$ replacing $\Sigma$.
	\end{enumerate}
\end{thm}
\begin{thm}[Breakdown of analyticity of the Hartree scattering operators]\label{thm:main2}
	Let $d\geq 2$ and $\frac{4}{3} < \gamma < 2$.
	Let $\mcal{T}\in\{\mcal{S},\mcal{W}\}$.
	\begin{enumerate}
		\item Let $s > \frac{5 + 5\gamma}{3+\gamma}$.
		Then $\mcal{T}:\Sigma\to L^2$ admits no extension to a map $L^2\to L^2$ which is H\"older continuous of order $s$ on any ball containing $0\in L^2$.
		\item Let $s > \frac{4+4\gamma}{2+\gamma}$.
		Then there exists a radius $R$ such that for any ball $B \subset B_R(0) \subset \Sigma$ (not necessarily containing the origin), $\mcal{T}:B\to L^2$ admits no extension to a map $L^2\to L^2$ which is H\"older continuous of order $s$ at any point in $B\cap L^2$.
	\end{enumerate}
\end{thm}
Theorems \ref{thm:main1} and \ref{thm:main2} are direct analogues of Theorem \ref{thm:pNLS_main} for the Hartree equation.
Theorem \ref{thm:main1} states that the scattering problem for HNLS is analytically well-posed for initial/final data in $\Sigma$ and $\mcal{F}H^1$.
The analyticity is consistent with the fact that the Hartree nonlinearity $F(u) = (|x|^{-\gamma}*|u|^2)u$ depends analytically on the solution $u$.
We isolate Theorem \ref{thm:main1} as a separate result from Theorem \ref{thm:main2} because we consider it to be one of independent interest for the scattering theory of HNLS.
To our knowledge, Theorem \ref{thm:main2} is the first analyticity result for the Hartree scattering operators in the mass-subcritical case.

Theorem \ref{thm:main2} states that despite Theorem \ref{thm:main1}, which says that the scattering problem in $\Sigma$ for HNLS is as well-posed as it can possibly be, the analogous problem in $L^2$ exhibits at best a finite amount of regularity with respect to the initial/final data.
We note that for $\frac{4}{3}<\gamma<2$, $\frac{5+5\gamma}{3+\gamma} < \frac{4+4\gamma}{2+\gamma}$.
Thus the breakdown of regularity we obtain is more severe at the origin than elsewhere.
In particular, we see that $\mcal{T}:\Sigma\to L^2$ has no $C^3$ extension to a map $L^2\to L^2$.
The lower range of $s$ for which we obtain failure of H\"older continuity is $s > \frac{5+5\gamma}{3+\gamma} > \frac{35}{13} \approx 2.69$ for part 1, and $s > \frac{4+4\gamma}{2+\gamma} > \frac{14}{5} =  2.8$ for part 2.

Theorems \ref{thm:main1} and \ref{thm:main2} improve on Theorem \ref{thm:pNLS_main} in the following ways:
\begin{enumerate}
	\item They comprise an \ita{infinite} loss of regularity between the scattering problems in $\Sigma$ and in $L^2$, whereas in Theorem \ref{thm:pNLS_main} the loss of regularity is finite in magnitude.
	This suggests that the smoothness of the nonlinearity does not play a significant role in the $\Sigma-L^2$ disparity in the scattering theory for mass-subcritical nonlinear Schr\"odinger equations.
	\item The expansion of $\mcal{T}$ and the breakdown of regularity are proved at points $u_0\neq 0$ as well.
	In Theorem \ref{thm:pNLS_main} the analogous claims are only proved at the origin, due to technical difficulties in working with the fractional power $|u|^pu$.
\end{enumerate}
However, we note that the gauge invariance of the nonlinearity does play an important role both in this paper and in \cites{Lee21}.
We have yet to investigate whether similar results can be proven for nonlinearities which are not gauge invariant.

\subsection{Past work}

Here we briefly review the history of the scattering theory for the defocusing mass-subcritical Hartree equation.
Results in this area roughly fall into the following categories, along which we organize our review:
\begin{enumerate}
	\item global well-posedness, asymptotic completeness, and existence of the wave operators on $\Sigma$, or more generally a weighted space $\Sigma^{\ell,m} = H^\ell\cap\mcal{F}H^m$.
	\item existence of scattering states in $L^2$ for initial data in weighted spaces $\Sigma^{\ell,m}$.
	\item regularity of the scattering operators.
	\item nonexistence of scattering states in $L^2$.
\end{enumerate}

The foundational work in the scattering theory of NLS equations with Hartree-type nonlinearity is that of Ginibre-Velo \cites{GiVe80}, which establishes local and global well-posedness and scattering in the weighted spaces $\Sigma^{\ell, 1}$, $\ell\geq 1$, under the assumption $2<\gamma<\min(4,d)$ (mass-supercritical, energy-subcritical).
Strauss \cites{Strauss81, Strauss81.sequel} extended this result to small-data scattering in $H^1$ for $2 \leq \gamma < \min(4,d)$ (recovering the mass-critical case and dropping the weighted assumption), and existence of wave operators on the Lebesgue space $L^{4d/(2d+\gamma)}\cap L^2$ for $\frac{4}{3} < \gamma < \min(4,d)$, which brings us partially into the mass-subcritical regime.
This was then extended by Hayashi-Tsutsumi \cites{HaTs87} in $d\geq 2$ to scattering in $\Sigma^{\ell,m}$ for $\ell,m\geq 1$ assuming $\frac{4}{3} < \gamma < \min(4,d)$.
Hayashi-Ozawa \cites{HaOz88} improved this in $d\geq 3$ to existence of wave operators on $\mcal{F}H^k$, $k\geq 1$ for $\frac{4}{3} < \gamma < 2$.
In these works, $\gamma = \frac{4}{3}$ is a exponent which plays a similar role to the Strauss exponent $p = \alpha(d)$ for pNLS: they are both threshold exponents for which the Gagliardo-Nirenberg inequality can be leveraged to obtain certain global-in-time spacetime bounds of the nonlinear evolution.
Nawa-Ozawa \cites{NaOz92} further improved this to existence of wave operators on $\mcal{F}H^k$, $k\geq 1$ for $d\geq 2$ and $1<\gamma<2$, recovering the full mass-subcritical range.
Finally, Masaki \cites{Masaki19} showed global well-posedness and scattering in the critically-scaling homogeneous weighted $L^2$ space $\mcal{F}\dot{H}^{s_c}$, where $s_c = 1-\frac{\gamma}{2}$ with $1 < \gamma < 2$, $d\geq 2$, assuming global spacetime bounds for the solution in $\mcal{F}\dot{H}^{s_c}$.

$\gamma = 1$ is an important threshold for the scattering theory of mass-subcritical HNLS.
Glassey \cites{Glassey77} showed that there can be no scattering theory in $L^2\cap L^1$ for $\gamma = 1$ and $d=3$ in the sense that the only solutions which are asymptotically free in $L^2\cap L^1$ with Schwartz data are trivial.
Hayashi-Tsutsumi \cites{HaTs87} improved this result, showing that in the full range $0<\gamma\leq 1$ and $d\geq 2$, the only solutions with $\Sigma^{\ell,m}$ data which are asymptotically free in $L^2$ are trivial.
Thus the mass-subcritical regime splits into two sub-regimes.
The first is the \ita{long-range} regime $0<\gamma\leq 1$, for which no $L^2$ scattering theory can exist.
In this range the conjectured behavior is \ita{modified scattering}, i.e. asymptotic convergence to a free evolution modulated by a phase, for which there is an abundance of available literature.
The second is the \ita{short-range} regime $1 < \gamma < 2$, the only part of the mass-subcritical regime where nontrivial asymptotically free solutions in $L^2$ are possible.

Once the existence of the scattering operators has been established, the next natural question is to determine their regularity.
For pNLS and HNLS, earlier investigations into this question include the work of Kita \cites{Kita03}, Kita-Ozawa \cites{KiOz05}, Carles-Ozawa \cites{CaOz08}, and Masaki \cites{Masaki09}, as well as the author's own result in \cites{Lee21}.
Analyticity for the HNLS scattering operators was first proved by Miao-Wu-Zhang \cites{MiWuZh09} in the case $d=3$, with  $H^1$ as the scattering topology and with a mass-supercritical, energy-subcritical nonlinearity (the same as that of Ginibre-Velo \cites{GiVe80}).
Carles-Gallagher \cites{CaGa09} established a general abstract framework for proving the analyticity of scattering operators for semilinear dispersive PDEs with real analytic nonlinearities.
They applied this framework to prove analyticity of the HNLS scattering operators on $\Sigma$ and/or $H^1$ in a mass-supercritical, energy-subcritical regime with $d\geq 3$, as well as analyticity for the pNLS scattering operators on $\Sigma$ and/or $H^1$ under various conditions on the dimension and for certain even values of $p$.

It is possible to obtain a wider range of scattering results by relaxing the topology of asymptotic convergence.
Hayashi-Tsutsumi \cites{HaTs87} proved that for $d\geq 2$ and $1<\gamma<\min(4,d)$, for any initial data $u_0\in\Sigma$ the corresponding global solution $u(t)$ scatters to a unique scattering state $u_+\in L^2$: in other words, the initial-to-scattering-state operator $\mcal{S}$ is well-defined as a map $\Sigma\to L^2$.
This result is sharp in light of the impossibility of an $L^2$ scattering theory for $\gamma \leq 1$.
This was then improved by Hayashi-Ozawa \cites{HaOz88}, who showed that $\mcal{S}$ is well-defined as a map $\mcal{F}H^1\to L^2$ for $1 < \gamma < \min(2,d)$.
As for the problem of the wave operators, Holmer-Tzirakis \cites{HoTz10} showed that for $d=2$ and $1 < \gamma < 2$, for any $H^1$ scattering state $u_+$ there exists a global $H^1$ solution $u$ which scatters to $u_+$.
However, because this global solution $u$ is not known to be uniquely determined by $u_+$, this falls short of defining the wave operator.

\subsection{Outline of the paper}

In Section \ref{sec:notation} we introduce the notation and basic estimates used throughout the paper.

Section \ref{sec:analyticity} is devoted to the proof of Theorem \ref{thm:main1}.
The proof is largely along the lines of the abstract framework of Carles-Gallagher \cites{CaGa09}.
The relevant estimates for our case arise as a consequence of the mass-subcritical scattering theory developed in \cites{HaTs87,HaOz88,NaOz92}.

In Section \ref{sec:breakdown} we prove Theorem \ref{thm:main2}.
The proof of part (1) proceeds largely along the lines of the proof of the analogous statement for pNLS given in \cites{Lee21}.
We perform a Taylor expansion of the integral form of the scattering operators, which at the origin takes the form
\[
	\mcal{T}(\phi) = \phi \pm i\int_0^\infty e^{-is\Delta} F(e^{is\Delta})~ds + e(\phi),	
\]
where $e(\phi)$ is an error term.
The idea is to identify the source of the breakdown of regularity in the main term
\[
	i\int_0^\infty e^{-is\Delta} F(e^{is\Delta}\phi)~ds,
\]
and to show that the error term is sufficiently negligible that the bad behavior in the main term manifests.
The failure of the following ``nonlinear free energy'' estimate plays a key role in the argument:
\[
	\int_0^\infty Q(e^{it\Delta}\phi)~dt \lesssim \|\phi\|_2^\alpha, ~\alpha>0,
\]
where $Q$ is the potential energy functional for HNLS \eqref{eqn:HNLS_potential_energy}.
It is easily seen by scaling that such control only holds for the mass-critical HNLS ($\gamma = 2$).
We note that the failure of such an estimate was also responsible in \cites{Lee21} for the breakdown of regularity of the scattering operators in the pNLS case, with $Q$ replaced by the potential energy functional for pNLS \eqref{eqn:pNLS_potential_energy}.
We speculate that the nonlinear free energy may be a governing quantity in the scattering theory for semilinear dispersive equations in general.
This idea is already implicit in the work of Cazenave-Weissler \cites{CaWe92}, in which the $L_t^1$-norm of the potential energy was used to define a notion of ``rapid decay'' and a scattering criterion for pNLS.

The proof of part (2) of Theorem \ref{thm:main2} is similar in spirit.
The source of the breakdown of regularity is still essentially in the failure of the nonlinear free energy estimate, and this manifests as a failure to control the third derivative term in the Taylor expansion of $\mcal{T}$ in $L^2$.
However, a modification of this estimate is needed because the ``resonant'' part of the main term now looks like
\[
	i\int_0^\infty e^{-is\Delta} F(w_1(\phi))~ds	
\]
where $w_1(\phi) = e^{it\Delta}\phi + \mcal{N}(u,u,w_1)$ for some nonlinear operator $\mcal{N}$, which complicates the estimate.
We get around this by rescaling in time in such a way that $w_1$ does behave like a free evolution for almost all times, and the error incurred by neglecting the remaining times is negligible.
Additional difficulties arise because the third derivative term contains extra nonresonant terms, which are not present in the expansion around $u_0 = 0$.
The conclusion is also weaker on balls in $L^2$ not containing $0$ because the error estimate is weaker: the series expansion around $0$ contains no fourth derivative term, so the error term $e(\phi)$ is quintic in $\phi$, whereas the expansion around $u_0\neq 0$ does contain fourth derivative terms and therefore the error is quartic.


\section{Notation and preliminary estimates}\label{sec:notation}

Let $X$ and $Y$ be two quantities.
We write $X\lesssim Y$ if there exists a constant $C>0$ such that $X\leq CY$.
If $C$ depends on parameters $a_1,\ldots,a_n$, i.e. $C = C(a_1,\ldots,a_n)$ and we wish to indicate this dependence, then we will write $X\lesssim_{a_1,\ldots,a_n} Y$.
If $X\lesssim Y$ and $Y\lesssim X$, we write $X\sim Y$.
If the constant $C$ is small, then we write $X \ll Y$.
We also employ the asymptotic notation $\mcal{O}(f)$ with its standard meaning.

We adopt the Japanese bracket notation $\langle x\rangle = (1+|x|^2)^\frac{1}{2}$.

We will be working with the mixed spacetime Lebesgue spaces $L_t^qL_x^r(I\times\bb{R}^d)$, with norms
\[
    \|u\|_{L_t^qL_x^r(I\times\bb{R}^d)} = \left( \int_I \left( \int_{\bb{R}^d} |u(t,x)|^r~dx\right)^{\frac{q}{r}}~dt\right)^{\frac{1}{q}}.    
\]
We will abbreviate the norm as $\|u\|_{L_t^qL_x^r(I\times\bb{R}^d)} = \|u\|_{L_t^qL_x^r(I)}$.
When $I$ is clear from context we will further abbreviate the norm as $\|u\|_{q,r}$.
For purely spatial integration, we write $\|f\|_{L^r(\bb{R}^d)} = \|f\|_r$.
For $1\leq r \leq \infty$, we denote by $r'$ the H\"older conjugate: $1 = \frac{1}{r} + \frac{1}{r'}$.
We will also occasionally use the mixed Lorentz-Lebesgue spaces $L_t^{q,p}L_x^r(I\times\bb{R}^d)$, where for $1\leq p<\infty$, $L_t^{q,p}(I)$ denotes the Lorentz space defined by the quasinorm
\[
	\|f\|_{L_t^{q,p}(I)} = q^{\frac{1}{p}}\left( \int_0^\infty t^p |\{s\in I:|f(s)| \geq t\}|^{\frac{p}{q}}~\frac{dt}{t} \right)^{\frac{1}{q}},
\]
and $L^{q,\infty}$ denotes weak $L^q$.

We recall the following fundamental estimates for the Schr\"odinger equation.
\begin{prop}[Dispersive estimate]\label{prop:dispersive_estimate}
    Let $2\leq r \leq\infty$.
    Then for all $t\neq 0$,
    \[
        \|e^{it\Delta}\phi\|_{L_x^r(\bb{R}^d)} \lesssim_{r,d} |t|^{\frac{d}{2} - \frac{d}{r}}\|\phi\|_{L^{r'}(\bb{R}^d)}.
    \]
\end{prop}
\begin{dfn}[Admissible pair]\label{dfn:admissible}
    Let $d\geq 1$ and $2 \leq q,r \leq \infty$.
    We say that $(q,r)$ is an \ita{admissible pair} if it satisfies the scaling relation $\frac{2}{q} + \frac{d}{r} = \frac{d}{2}$ and $(d,q,r)\neq(2,2,\infty)$.
    We say that $(\alpha,\beta)$ is a \ita{dual admissible pair} if $(\alpha',\beta')$ is an admissible pair.
\end{dfn}
\begin{prop}[Strichartz estimates]\label{prop:Strichartz}
    Let $d\geq 1$, let $(q,r)$ be an admissible pair, and let $(\alpha,\beta)$ be a dual admissible pair.
    Then for any interval $I\subset\bb{R}$,
    \[
        \|e^{it\Delta}\phi\|_{L_t^qL_x^r(I\times\bb{R}^d)} \lesssim \|\phi\|_{L^2(\bb{R}^d)},
    \]
    \[
        \left\| \int_\bb{R} e^{-is\Delta} F(s)~ds \right\|_{L^2(\bb{R}^d)} \lesssim \|F\|_{L_t^\alpha L_x^\beta(\bb{R}\times\bb{R}^d)},
    \]
    \[
        \left\| \int_{s<t} e^{i(t-s)\Delta} F(s)~ds \right\|_{L_t^qL_x^r(\bb{R}\times\bb{R}^d)} \lesssim \|F\|_{L_t^\alpha L_x^\beta(\bb{R}\times\bb{R}^d)}.    
    \]
	Also, the following Lorentz space versions of these estimates hold for $2<q<\infty$:
    \[
        \|e^{it\Delta}\phi\|_{L_t^{q,2}L_x^r(I\times\bb{R}^d)} \lesssim \|\phi\|_{L^2(\bb{R}^d)},
    \]
    \[
        \left\| \int_{s<t} e^{i(t-s)\Delta} F(s)~ds \right\|_{L_t^{q,2}L_x^r(\bb{R}\times\bb{R}^d)} \lesssim \|F\|_{L_t^{q',2} L_x^{r'}(\bb{R}\times\bb{R}^d)}.    
    \]
\end{prop}
We will make use of the vector field $J(t) = x + 2it\nabla$, which is standard in the scattering theory of Schr\"odinger equations.
$J$ obeys the identity
\begin{equation}\label{eqn:J(t)_identity}
	J(t) = M(t)(2it\nabla)M(-t) = e^{it\Delta}xe^{-it\Delta}
\end{equation}
where $M(t) = e^{i|x|^2/4t}$; it measures the evolution of the center of mass for free evolutions.
It is associated to the following decay estimate:
\begin{lem}[\cites{NaOz92}]\label{lem:commutator_decay}
	For $2 \leq r < \frac{2d}{d-2}$ and $t\neq 0$, we have
	\[
		\|u(t)\|_r \lesssim_{d,r} |t|^{-\theta(d,r)}\|u(t)\|_2^{1-\theta(d,r)}\|J(t)u(t)\|_2^{\theta(d,r)},
	\]
	where $\theta(d,r) = \frac{d(r-2)}{2r}$.
\end{lem}
\begin{proof}
	By the decomposition $J(t) = M(t)(2it\nabla)M(-t)$ \eqref{eqn:J(t)_identity} and the Gagliardo-Nirenberg inequality,
	\begin{align*}
		\|u(t)\|_r
			&= \|M(-t)u(t)\|_r
			\lesssim \|u(t)\|_2^{1-\theta(d,r)}\|\nabla M(-t)u(t)\|_2^{\theta(d,r)}\\
			&= |t|^{-\theta(d,r)}\|u(t)\|_2^{1-\theta(d,r)}\|J(t)u(t)\|_2^{\theta(d,r)}. \qedhere	
	\end{align*}
\end{proof}
We note that for the special case $r = \frac{4d}{2d-\gamma}$, $\theta = \frac{\gamma}{4}$; from this point on we fix this as the value of $\theta$.

Lastly, we will need some preliminary estimates on the nonlinearity.
Define
\begin{equation}\label{eqn:trilinear_form}
	T(u,v,w) = (|x|^{-\gamma}*(u\ol{v}))w.	
\end{equation}
Applications of H\"older's inequality and the Hardy-Littlewood-Sobolev inequality yield the following multilinear estimates:
\begin{lem}[Nonlinear estimates \cites{HaTs87}]\label{lem:multilinear_estimate}
	Let $0 < \gamma < d$ and $r = \frac{4d}{2d-\gamma}$.
	Then for all $u,v,w\in L^r(\bb{R}^d)$,
	\[
		Q(u)
		= \int_{\bb{R}^d} (|x|^{-\gamma}*|u|^2)|u(x)|^2dx
		\lesssim \|u\|_r^4,
	\]
	\[
		\|T(u,v,w)\|_{r'} 
			\lesssim \|u\|_r\|v\|_r\|w\|_r,
	\]
	\[
		\|\nabla T(u_1,u_2,u_3)\|_{r'} 
			\lesssim \sum_{i=1}^3 \|\nabla u_i\|_r \prod_{j\neq i} \|u_j\|_r,
	\]
	\[
		\|J(t) T(u_1,u_2,u_3)\|_{r'} 
			\lesssim \sum_{i=1}^3 \|J(t) u_i\|_r \prod_{j\neq i} \|u_j\|_r.
	\]
\end{lem}


\section{Analyticity of the Hartree scattering operators}\label{sec:analyticity}

For the remainder of this paper, we assume $d\geq 2$ and $\frac{4}{3}<\gamma<2$.

Our goal in this section is to prove Theorem \ref{thm:main1}.
As we have mentioned in the introduction, this proceeds largely along the lines of the framework set out in \cites{CaGa09}, adapted to the estimates we have for the mass-subcritical Hartree equation.

We first address the analyticity of the wave operator.
Let $u_+\in\Sigma$ be a scattering state, and let $v\in\Sigma$ with $\|v\|_\Sigma = 1$ be arbitrary.
By \cites{HaTs87}, under our current assumptions there exists a unique global solution $u\in C_t\Sigma(\bb{R})$ to Equation \eqref{eqn:HNLS} which scatters to $u_+$, and for each $\ve>0$ there exists a unique global solution $u^\ve\in C_t\Sigma$ which scatters to $u_+ + \ve v$.
Moreover, the wave operator $\mcal{W}:\Sigma\to\Sigma$ is well-defined.

Write $u^\ve = u + w^\ve$.
Our goal is to show that for $\|u_+\|_\Sigma$ sufficiently small, $w^\ve$ admits the norm-convergent expansion
\[
	w^\ve(t) = \sum_{k=1}^\infty \ve^k w_k(t) ~\tnm{as}~\ve\to 0,
\]
where $(w_k)$ are elements of an appropriate function space determined by contraction mapping.
This argument consists of three parts:
\begin{enumerate}
	\item determining the hierarchy of equations satisfied by the sequence $(w_k)$;
	\item showing that $(w_k)$ is sufficiently strongly bounded in a global spacetime norm, so that the series for $u^\ve$ is norm convergent;
	\item showing that the series for $w^\ve$ does actually converge to $w^\ve$.
\end{enumerate}
It will emerge as a consequence that $\mcal{W}$ admits the norm-convergent expansion
\[
	\mcal{W}(u_++\ve v) = \sum_{j=0}^\infty \ve^kv_k~\tnm{as}~\ve\to 0
\]
where $(v_k)\subset\Sigma$.

\subsection{Hierarchy equations}

The coefficients $(w_k)$ of the series for $u^\ve$ formally satisfy a hierarchy of coupled PDEs.
We express $u^\ve$ in integral form, then match like powers of $\ve$ to obtain the coefficients.
Let us write
\[
	\mcal{N}(u,v,w)(t) = i\int_t^\infty e^{i(t-s)\Delta} T(u,v,w)~ds
\]
where $T$ is the trilinear form defined in \eqref{eqn:trilinear_form}.
Matching zero-th order terms in $\ve$ yields
\[
	w_0(t) = u(t) = e^{it\Delta}u_+ + \mcal{N}(u,u,u)(t)
\]
Matching first order terms in $\ve$ yields
\[
	w_1(t) = e^{it\Delta}v + \mcal{N}(u,u,w_1)(t) + \mcal{N}(u,w_1,u)(t) + \mcal{N}(w_1,u,u)(t).
\]
Higher-order terms behave similarly, involving symmetric sums of the trilinear operators $T$ with arguments in $\{u,w_1,w_2,\ldots\}$.
To simplify notation, we introduce the symmetric sum operator $S$ which sums over all distinct permutations of the ordered triple $(u,v,w)$.
For example,
\[
	S\mcal{N}(u,u,w_1) = \mcal{N}(u,u,w_1) + \mcal{N}(u,w_1,u) + \mcal{N}(w_1,u,u).	
\]
Such a symmetric sum has either one, three, or six summands.
With this notation, the full hierarchy of equations for the coefficients takes the following form:
\begin{align}
		w_0(t) &= e^{it\Delta}u_+ + \mcal{N}(u,u,u)(t),\\
		w_1(t) &= e^{it\Delta}v + S\mcal{N}(u,u,w_1)(t),\\
		w_N(t) &= \sum_{j+k+\ell = N} \mcal{N}(w_j,w_k,w_\ell)(t), ~N\geq 2.
\end{align}

\subsection{Coefficient estimates}
Fix $r = \frac{4d}{2d-\gamma}$ and $q = \frac{8}{\gamma}$; then $(q,r)$ is a Schr\"odinger-admissible pair.
Also fix $\alpha = \frac{8}{4-\gamma}$.
With these choices we have $\frac{1}{q'} = \frac{1}{q} + \frac{2}{\alpha}$.

For a time interval $I$, we define the space $Y(I)$ via its norm
\[
	\|f\|_{Y(I)} = \|f\|_{L_t^\infty L_x^2(I)} + \|f\|_{L_t^qL_x^r(I)}.
\]
We define the space $X(I)$ by the norm
\[
	\|f\|_{X(I)} = \|f\|_{Y(I)} + \|J(t)f\|_{Y(I)} + \|\nabla f\|_{Y(I)}.	
\]
\begin{rem}\label{rem:Lorentz}
	$X(I)$ is adapted to the $\Sigma$-norm and is thus used to prove part (1) of Theorem \ref{thm:main1}.
	The results of this section can also be proved in the $\mcal{F}H^1$-adapted space $Z(I)$ defined by the norm 
	\[
		\|f\|_{Z(I)} = \|f\|_{Y(I)} + \|J(t)f\|_{Y(I)}.
	\]
	We leave it to the reader to verify that all of the relevant estimates hold with $Z(I)$ replacing $X(I)$, thus obtaining part (2) of Theorem \ref{thm:main1} as well.

	The results of this section also hold in the Lorentz-modified space $X^*(I)$, defined analogously to $X(I)$ but replacing $Y(I)$ by $Y^*(I)$, where
	\[
		\|f\|_{Y^*(I)} = \|f\|_{L_t^\infty L_x^2(I)} + \|f\|_{L_t^{q,2}L_x^r(I)}.
	\]
	Since $L_t^{q,2}$ is normable for our choice of $q$, these are still Banach spaces.
	Again, we leave it to the reader to check that all of the results we prove in this section can be adapted to the $X^*$ setting as well: the key estimate is
	\[
		\|T(u,v,w)\|_{L_t^{q',2}L_x^{r'}} 
			\lesssim \|u\|_{L_t^{\alpha,\infty}L_x^r}	\|v\|_{L_t^{\alpha,\infty}L_x^r} \|w\|_{L_t^{q,\infty}L_x^r}
	\]
	and the analogous estimates for $J(t)T(u,v,w)$ and $\nabla T(u,v,w)$, which follow from Lemma \ref{lem:multilinear_estimate} and H\"older's inequality for Lorentz spaces.
	The Lorentz space refinement will become relevant in Section \ref{sec:breakdown}.
\end{rem}
We construct the power series expansions of the wave and scattering operators by constructing the coefficients $(w_k)$ on the interval $[0,\infty)$ and then taking the appropriate limits in $t$.
For notational convenience we construct $(w_k)$ first on $[1,\infty)$.
Composing with the time-translation symmetry of HNLS then gives us the coefficients on $[0,\infty)$.
\begin{prop}\label{prop:coefficient_estimates}
	For any $u_+ \in \Sigma$ and any $v\in\Sigma$ with $\|v\|_\Sigma = 1$, there exists a constant $\Lambda = \Lambda(R,d,\gamma) > 0$ such that for all $k\geq 1$,
	\[
		\|w_k\|_{X([1,\infty))} \leq a_k\Lambda^k,
	\]
	where $(a_k)$ is a sequence of positive numbers satisfying
	\[
		a_k \lesssim (C_0a_1)^k
	\]
	for some positive constant $C_0$.
\end{prop}
\begin{cor}
	Under the same hypotheses, the series
	\[
		\sum_{k=1}^\infty \ve^k w_k	
	\]
	converges in the norm topology of $X(\bb{R})$ for all sufficiently small $\ve>0$.
\end{cor}
\begin{lem}[\cites{Kishimoto18, BeTa06}]\label{lem:sequential_gronwall}
	Let $(a_j)$ be a sequence of positive numbers satisfying
	\[
		a_N \leq C\sum_{\substack{j+k+\ell=N \\ j,k,\ell\neq N}} a_ja_ka_\ell, ~N\geq 2.
	\]
	Then there exist constants $C_0, C_1>0$ such that
	\[
		a_N \leq C_1(C_0a_1)^N	
	\]
	for all $N\geq 1$.
\end{lem}
\begin{rem}
	Our treatment differs from that of \cites{CaGa09} in the use of this lemma.
\end{rem}
\begin{proof}
	We claim the stronger inequality
	\begin{equation}\label{eqn:stronger_sequential_gronwall}
		\langle N\rangle^2 a_N \leq C_1(C_0a_1)^N.	
	\end{equation}
	We proceed by induction on $N$.

	First we assume $C_1C_0 \geq 1$.
	Under this assumption, the base case $N=1$ is trivial.

	Now assume \eqref{eqn:stronger_sequential_gronwall} holds for $1,\ldots,N-1$.
	We estimate:
	\begin{align*}
		\langle N\rangle^2 a_N
			&\leq C\sum_{\substack{j+k+\ell=N \\ j,k,\ell\neq N}} a_ja_ka_\ell\langle j+k+\ell\rangle^2\\
			&\leq CC_1^3(C_0a_1)^N \sum_{\substack{j+k+\ell=N \\ j,k,\ell\neq N}} \frac{\langle j+k+\ell\rangle^2}{\langle j\rangle^2\langle k\rangle^2\langle \ell\rangle^2}\\
			&\leq 3CC_1^3(C_0a_1)^N \sum_{\substack{j+k+\ell=N \\ j,k,\ell\neq N}} \frac{\langle j\rangle^2 + \langle j\rangle^2 + \langle j\rangle^2}{\langle j\rangle^2\langle k\rangle^2\langle \ell\rangle^2}\\
			&\leq 9CC_1^3(C_0a_1)^N \sum_{\substack{j+k+\ell=N \\ j,k,\ell\neq N}} \frac{\langle j\rangle^2}{\langle j\rangle^2\langle k\rangle^2\langle \ell\rangle^2}.
	\end{align*}
	The remaining sum we bound as follows:
	\begin{align*}
		\sum_{\substack{j+k+\ell=N \\ j,k,\ell\neq N}} \frac{\langle j\rangle^2}{\langle j\rangle^2\langle k\rangle^2\langle \ell\rangle^2}
			\leq \sum_{\ell=0}^N \sum_{k=0}^{N-\ell} \frac{1}{\langle k\rangle^2\langle \ell\rangle^2}
			\leq \left(\sum_{k=0}^\infty \langle k\rangle^{-2}\right)^2 = C_2^2.
	\end{align*}
	Therefore
	\[
		\langle N\rangle^2 a_N \leq (9CC_1^2C_2^2)C_1(C_0a_1)^N.
	\]
	The claim then follows by choosing $C_1 = (9CC_2^2)^{-\frac{1}{2}}$.

	Finally, we note that once $C_1$ is fixed as above, we are free to choose $C_0$ as large as we like in \eqref{eqn:stronger_sequential_gronwall}.
	Thus we can always assume $C_1C_0\geq 1$, justifying our earlier assumption.
\end{proof}
\begin{proof}[Proof of Proposition \ref{prop:coefficient_estimates}]
	We proceed by induction on $k$.

	Take $k=1$.
	Let $I$ be a time interval.
	By Strichartz, Lemma \ref{lem:multilinear_estimate}, and H\"older in time, we find that
	\begin{align*}
		\|1_{t\in I}w_1\|_{\infty,2} 
			&\leq \|1_{t\in I}e^{it\Delta}v\|_{\infty,2} + C(d,\gamma)\|S\mcal{N}(u,u,w_1)\|_{\infty,2}\\
			&\leq \|1_{t\in I}e^{it\Delta}v\|_{\infty,2} + C(d,\gamma)\|1_{t\in I}u\|_{\alpha,r}^2\|1_{t\in I}w_1\|_{q,r}.
	\end{align*}
	By Lemma \ref{lem:commutator_decay} and our choice of $r$ and $\alpha$,
	\begin{align*}
		\|1_{t\in I}u\|_{\alpha,r} 
			&\lesssim \|1_{t\in I}u\|_{\infty,2}^{1-\theta}\|1_{t\in I}J(t)u\|_{\infty,2}^\theta \left( \int_I |t|^{-2\gamma/(4-\gamma)} ~dt \right)^{1/\alpha} \\
			&\leq \|1_{t\in I}u\|_{X([1,\infty))}\left( \int_I |t|^{-2\gamma/(4-\gamma)} ~dt \right)^{1/\alpha}.
	\end{align*}
	Since $\gamma>\frac{4}{3}$, $\frac{2\gamma}{4-\gamma}>1$.
	Since $u\in X([1,\infty))$, we can decompose $[1,\infty)$ into a union of finitely many disjoint intervals $I_k$ such that
	\[
		\|1_{t\in I_k} w_1\|_{\infty,2} \leq \|1_{t\in I_k}e^{it\Delta}v\|_{\infty,2} + \frac{1}{12} \|1_{t\in I_k}w_1\|_{X([1,\infty))}.
	\]
	Arguing similarly for the remaining parts of the $X(I)$ norm, we find that
	\[
		\|1_{t\in I_k}w_1\|_{X([1,\infty))} \leq \|1_{t\in I_k} e^{it\Delta}v\|_{X([1,\infty))} + \frac{1}{2}\|1_{t\in I_k}w_1\|_{X([1,\infty))}.
	\]
	Since $\|1_{t\in A\cup B}f\|_{X([1,\infty))} \sim \|1_{t\in A} f\|_{X([1,\infty))} + \|1_{t\in B}f\|_{X([1,\infty))}$ whenever $A$ and $B$ are disjoint, we conclude by Strichartz and \eqref{eqn:J(t)_identity} that
	\[
		\|w_1\|_{X([1,\infty))} \lesssim_{d,\gamma} \|e^{it\Delta}v\|_{X([1,\infty))} \lesssim\|v\|_\Sigma  = 1.
	\]
	Let $C(d,\gamma)$ be the implicit constant in this estimate, and set $\Lambda = C(d,\gamma)$.
	This establishes the case $k=1$.

	Now define the sequence $(a_N)$ by $a_1 = C(d,\gamma)$ from above, and
	\[
		a_N = C'(d,\gamma)\sum_{\substack{j+k+\ell=N \\ j,k,\ell\neq N}} a_ja_ka_\ell,
	\]
	where $C'(d,\gamma)$ is a constant to be determined.
	Assume the bound
	\[
		\|w_j\|_{X([1,\infty))}	\leq a_j\Lambda^j
	\]
	for $j=1,\ldots,N-1$, and consider $w_N$.
	Working as in the previous case, we find that
	\begin{align*}
		\|1_{t\in I} w_N\|_{X([1,\infty))} 
			&\lesssim \sum_{\substack{j+k+\ell=N \\ j,k,\ell\neq N}} \|1_{t\in I} S\mcal{N}(w_j,w_k,w_\ell) \|_{X([1,\infty))} \\
			&~~~~+ C(I)^2\|1_{t\in I}w_0\|_{X([1,\infty))}^2\|1_{t\in I} w_N\|_{X([1,\infty))},
	\end{align*}
	where
	\[
		C(I) = \left( \int_I |t|^{-2\gamma/(4-\gamma)} ~dt \right)^{1/\alpha}.
	\]
	Once again, this implies that
	\[
		\|w_N\|_{X([1,\infty))}
			\lesssim \sum_{\substack{j+k+\ell=N \\ j,k,\ell\neq N}} \|S\mcal{N}(w_j,w_k,w_\ell) \|_{X([1,\infty))}.
	\]
	By Strichartz, Lemma \ref{lem:multilinear_estimate}, H\"older, Lemma \ref{lem:commutator_decay}, and invoking the induction hypothesis, we find that
	\[
		\|w_N\|_{X[1,\infty)}
			\lesssim \Lambda^N \sum_{\substack{j+k+\ell=N \\ j,k,\ell\neq N}} a_ja_ka_\ell.	
	\]
	The implicit constant in this estimate can be defined independently of $N$.
	Thus if we set $C'(d,\gamma)$ to be this constant, then we arrive at
	\[
		\|w_N\|_{X([1,\infty))}
			\lesssim a_N\Lambda^N.
	\]
	Invoking Lemma \ref{lem:sequential_gronwall} to control the growth of $(a_N)$ completes the proof.
\end{proof}

\subsection{Convergence}

We have shown that the series
\[
	\sum_{k\geq 1} \ve^k w_k	
\]
is norm convergent in the space $X([1,\infty))$ for all sufficiently small $\ve>0$.
Our next goal is to show that it converges to the correct object.
\begin{prop}\label{prop:convergence_wave_ops}
	Let $\ve>0$ be such that $\sum_{k\geq 1} \ve^k w_k$ converges in $X([1,\infty))$.
	Then
	\[
		\left\| u^\ve - u - \sum_{k=1}^{N-1} \ve^k w_k \right\|_{X([1,\infty))} \to 0 ~\tnm{as}~N\to\infty.	
	\]
\end{prop}
\begin{proof}
	Let $W_{\geq N} = u^\ve - u - \sum_{k=1}^{N-1} \ve^k w_k = u^\ve - u - W_{<N}$.
	Then for $N\geq 2$, $W_{\geq N}$ satisfies the equation
	\begin{align*}
		W_{\geq N}(t) 
			&= \mcal{N}(u^\ve,u^\ve,u^\ve) - \mcal{N}(u,u,u) - \sum_{M=1}^{N-1} \ve^M\sum_{j+k+\ell = M} S\mcal{N}(w_j,w_k,w_\ell)\\
			&= i\int_t^\infty e^{i(t-s)\Delta} G(u,W_{<N},W_{\geq N})~ds + \mcal{N}(W_{<N},W_{< N}, W_{<N}) \\
			&\hspace{8em} - \sum_{M=1}^{N-1} \ve^M\sum_{j+k+\ell = M} S\mcal{N}(w_j,w_k,w_\ell)
	\end{align*}
	where $G(u,W_{<N},W_{\geq N})$ consists of all the terms $T(a,b,c)$ with at least one argument equal to $W_{\geq N}$.
	Arguing as in the proof of Proposition \ref{prop:coefficient_estimates}, we can decompose $[1,\infty)$ into a finite collection of disjoint intervals $I_k$ such that
	\[
		\left\| 1_{t\in I_k}\int_t^\infty e^{i(t-s)\Delta} G(u,W_{<N},W_{\geq N})~ds \right\|_{X([1,\infty))} \leq \frac{1}{2}\|1_{t\in I_k} W_{\geq N}\|_{X([1,\infty))}.
	\]
	For the remaining terms, we write
	\begin{align*}
		\mcal{N}(W_{<N},W_{<N},W_{<N})
			&= \sum_{M=1}^{N-1} \ve^M\sum_{j+k+\ell = M} S\mcal{N}(w_j,w_k,w_\ell)\\
			&\hspace{2em} + \sum_{\substack{ 1 \leq j,k\leq N-1 \\ j+k\geq N}} \ve^{j+k} S\mcal{N}(u,w_j,w_k)\\
			&\hspace{2em} + \sum_{\substack{ 1\leq j,k,\ell\leq N-1 \\ j+k+\ell \geq N}} \ve^{j+k+\ell} S\mcal{N}(w_j,w_k,w_\ell).
	\end{align*}
	The first term cancels exactly with the remaining terms in the previous expression for $W_{\geq N}$.
	For the latter two terms, by Proposition \ref{prop:coefficient_estimates} we have
	\[
		\left\| \sum_{\substack{ 1 \leq j,k\leq N-1 \\ j+k\geq N}} \ve^{j+k} S\mcal{N}(u,w_j,w_k)~ds \right\|_{X([1,\infty))}
			\lesssim (\ve\Lambda)^N
	\]
	as long as $\ve\Lambda < 1$, and similarly
	\[
		\left\| \sum_{\substack{ 1\leq j,k,\ell\leq N-1 \\ j+k+\ell \geq N}} \ve^{j+k+\ell}S\mcal{N}(w_j,w_k,w_\ell)~ds \right\|_{X([1,\infty))}
			\lesssim (\ve\Lambda)^N.
	\]
	Thus we conclude that
	\[
		\|W_{\geq N}\|_{X([1,\infty))} \lesssim (\ve\Lambda)^N,
	\]
	and sending $N\to\infty$ proves the desired claim.
\end{proof}
Thus we have shown that the map $\Sigma\to X([1,\infty))$ that sends the scattering data $u_+$ to the solution $u:[1,\infty)\times\bb{R}^d\to\bb{C}$ depends analytically on $u_+$, and we can compute the coefficients of the power series expansion around any point $u_+$ inductively.
Consequently, the same holds for the map $\Sigma\to\Sigma: u_+\mapsto u(t=1)$.

A similar argument applies to proving the analyticity of the initial-to-scattering state operator $\mcal{S}:\Sigma\to\Sigma$.
For initial data $u\in\Sigma$ and $v\in\Sigma$ we define $u\in X([0,\infty))$ to be the global solution to HNLS with $u(0) = u_0$, and $u^\ve(t)$ to the global solution to HNLS with $u^\ve(0) = u_0 + \ve v$; these are well-defined by the existing scattering theory in $\Sigma$.
Then arguments similar to those of Propositions \ref{prop:coefficient_estimates} and \ref{prop:convergence_wave_ops} show that $u^\ve$ admits the expansion
\[
	u^\ve = u + \sum_{k\geq 1} \ve^k w_k	
\]
with $\|w_k\|_{X([0,\infty))} \lesssim \Lambda^k$ for some fixed $\Lambda$ and small $\ve$.
To finish the proof of Theorem \ref{thm:main1} it remains to show:
\begin{prop}\label{prop:convergence_scattering_ops}
	For all $k\geq 1$, the limits
	\[
		w_k^+ = \lim_{t\to\infty} e^{-it\Delta}w_k(t)	
	\]
	exist in $\Sigma$, and
	\[
		\mcal{S}(u_0 + \ve v) = u_+ + \sum_{k\geq 1} \ve^k w_k^+	
	\]
	with the latter series converging in $\Sigma$.
\end{prop}
\begin{proof}
	We claim that for each $N\geq 1$, $(e^{-it\Delta}w_N(t))_{t\geq 0}$ is Cauchy in $\Sigma$ as $t\to\infty$.
	We have
	\[
		e^{-it_1\Delta}w_N(t_1) - e^{-it_2\Delta}w_N(t_2)
			= 	\sum_{j+k+\ell = N} -i\int_{t_1}^{t_2} e^{-is\Delta} ST(w_j,w_k,w_\ell)~ds
	\]
	By \ref{eqn:J(t)_identity}, Lemma \ref{lem:multilinear_estimate}, Strichartz, and H\"older as before,
	\begin{align*}
		\|e^{-it_1\Delta}w_N(t_1) - e^{-it_2\Delta}w_N(t_2)\|_\Sigma 
			&\lesssim I(t_1,t_2)\sum_{j+k+\ell = N} \|w_j\|_{X([1,\infty))}\|w_k\|_{X([1,\infty))}\|w_\ell\|_{X([1,\infty))}
	\end{align*}
	where
	\[
		I(t_1,t_2) = \left( \int_{t_2}^{t_1} |t|^{-2\gamma/(4-\gamma)}~dt \right)^{2/\alpha}.
	\]
	Since the integral tends to $0$ as $t_1,t_2\to 0$, the claim follows.
	Therefore the sequence $(w_N^+)\subset\Sigma$ is well-defined, and for $N\geq 2$ we may write
	\[
		w_N^+ = \sum_{j+k+\ell = N} -i\int_0^\infty e^{-is\Delta} ST(w_j,w_k,w_\ell)~ds
	\]
	(with appropriate changes for $N=1$).
	Working as in the proof of Proposition \ref{prop:coefficient_estimates}, we can show that there exists a constant $\Lambda>0$ such that
	\[
		\|w_N^+\|_\Sigma \lesssim \Lambda^N.
	\]
	Therefore the series $\sum_k \ve^k w_k^+$ converges in $\Sigma$ for $\ve>0$ sufficiently small, and a similar argument to Proposition \ref{prop:convergence_wave_ops} shows that
	\[
		\mcal{S}(u_0 + \ve v) = u_+ + \sum_{k=1}^\infty \ve^k w_k^+	
	\]
	in the sense of convergence of the series in $\Sigma$ to the LHS.
\end{proof}


\section{Breakdown of analyticity}\label{sec:breakdown}

We now turn to the proof of Theorem \ref{thm:main2}.

\subsection{Breakdown at the origin}
We first consider part (1) of Theorem \ref{thm:main2}.
From here on, we abuse notation and redefine
\[
	\mcal{N}(u,v,w) = i\int_0^\infty e^{-is\Delta} T(u,v,w)~ds.	
\]

Let $\mcal{T}\in\{\mcal{S},\mcal{W}\}$, regarding it as a map $\Sigma\to L^2$, and consider its power series expansion $\mcal{T}(v) = \sum_{k\geq 1} w_k$ at $0\in\Sigma$ (for small $\|v\|_\Sigma$).
A careful accounting of the hierarchy of equations governing the coefficients shows that all even-indexed terms vanish, so
\[
	\mcal{T}(v) = v + \mcal{N}(e^{it\Delta}v,e^{it\Delta}v,e^{it\Delta}v) + \sum_{\substack{k\geq 5 \\ k~\tnm{odd}}} w_k.	
\]
To establish part (1) of Theorem \ref{thm:main2} it is enough to show:
\begin{prop}
	For any $s > \frac{5 + 5\gamma}{3+\gamma}$, we have
	\[
		\|\mcal{T}(v) - v\|_2 	\neq \mcal{O}_{L^2}(\|v\|_2^s).
	\]
\end{prop}
\begin{proof}
	Let $v\in\Sigma$ be sufficiently small so that the expansion $\mcal{T}(v) = \sum_k w_k$ holds.
	By $L^2$-duality, Fubini, and unitarity of the free propagator $e^{it\Delta}$ we have
	\begin{align*}
		\|\mcal{T}(v) - v \|_2
			&\geq \frac{1}{\|v\|_2} |\langle \mcal{N}(e^{it\Delta}v,e^{it\Delta}v,e^{it\Delta}v), v \rangle | \|e(v)\|_2\\
			&= \frac{1}{\|v\|_2} \int_0^\infty Q(e^{is\Delta}v)~ds - \|e(v)\|_2
	\end{align*}
	where
	\[
		e(v) = 	\sum_{\substack{k\geq 5 \\ k~\tnm{odd}}} w_k.
	\]
	For fixed nonzero $v\in\Sigma$ and $\ve,\sigma>0$, we define
	\[
		v_{\ve,\sigma}(x) = \frac{\ve}{\sigma^{\frac{d}{2}}}v\left( \frac{x}{\sigma}\right).
	\]
	We will show the existence of a sequence of parameters $(\ve,\sigma)$ such that:
	\begin{enumerate}
		\item $\|v_{\ve,\sigma}\|_\Sigma \ll 1$ (so that the series expansion holds for $\mcal{T}(v_{\ve,\sigma})$);
		\item taking the limit along the sequence $(\ve,\sigma)$, we have
		\[
			\lim_{(\ve,\sigma)}\frac{1}{\|v_{\ve,\sigma}\|_2^s}\left( \frac{1}{\|v_{\ve,\sigma}\|_2} \int_0^\infty Q(e^{is\Delta}v_{\ve,\sigma})~ds - \|e(v_{\ve,\sigma})\|_2 \right) \to \infty.
		\]
	\end{enumerate}
	Since $(v_{\ve,\sigma})$ is an $L^2$-bounded sequence, the claim immediately follows.

	We will take $\ve \ll 1$ and $\sigma \gg 1$ with $\ve\sigma \ll 1$.
	The last condition keeps us in the regime of small $\|v_{\ve,\sigma}\|_\Sigma$, so that the power series expansion continues to hold.
	Then the family $(v_{\ve\sigma})$ obeys the following scalings:
	\[
		\|v_{\ve,\sigma}\|_2 \sim \ve, ~\|v_{\ve,\sigma}\|_\Sigma \sim \ve\sigma, ~\|\nabla v_{\ve,\sigma}\|_2 \sim \ve\sigma^{-1}.
	\]
	Moreover, by the parabolic scaling symmetry $e^{it\Delta}v(x) \leftrightarrow e^{i\sigma^{-2}t\Delta}v(\sigma^{-1}x)$ of the free Schr\"odinger flow,
	\[
		\int_0^\infty Q(e^{is\Delta}v_{\ve,\sigma})~ds 
			= \ve^4\sigma^{2-\gamma}\int_0^\infty Q(e^{is\Delta}v)~ds \sim_{\|v\|_\Sigma} \ve^4\sigma^{2-\gamma}.
	\]
	Here, the finiteness of the integral follows from Lemmas \ref{lem:commutator_decay}, \ref{lem:multilinear_estimate}, and the Gagliardo-Nirenberg inequality to control $\|e^{is\Delta}v\|_r$ near $t=0$.
	
	By Proposition \ref{prop:coefficient_estimates}, the error term obeys the estimate
	\[
		\|e(v_{\ve,\sigma})\|_2 \lesssim \|v_{\ve,\sigma}\|_2^5 \sim \ve^5\sigma^5.	
	\]
	Therefore
	\[
		\|\mcal{T}(v_{\ve,\sigma}) - v_{\ve,\sigma} \|_2 \gtrsim \ve^3\sigma^{2-\gamma} - \ve^5\sigma^5.
	\]
	We now take $\ve = \sigma^{-j}$ for some $j>1$ we will choose momentarily; this guarantees that  $\|v_{\ve,\sigma}\|_\Sigma \sim \ve\sigma \ll 1$ as $\sigma \gg 1$.
	Then
	\[
		\ve^3\sigma^{2-\gamma} - \ve^5\sigma^5 = \sigma^{-3j + 2 - \gamma} - \sigma^{-5j + 5} \sim \sigma^{-3j + 2 - \gamma}
	\]
	as $\sigma\to\infty$ so long as $-3j + 2 - \gamma > -5j + 5$, which is equivalent to the condition $j > \frac{3+\gamma}{2}$.
	For such $j$ and $\sigma \gg 1$, we have
	\[
		\frac{1}{\|v_{\ve,\sigma}\|_2^s}\|\mcal{T}(v_{\ve,\sigma}) - v_{\ve,\sigma} \|_2 \gtrsim \sigma^{(s-3)j+2-\gamma}.
	\]
	The RHS tends to $\infty$ as $\sigma\to\infty$ provided that $(s-3)j + (2-\gamma) > 0$.
	When $s \geq 3$, this is automatically satisfied since $\gamma < 2$; when $s< 3$, it is equivalent to the condition $j < \frac{2-\gamma}{3-s}$.
	Therefore it suffices to find a $j$ satisfying
	\[
		\frac{3+\gamma}{2} < j < \frac{2-\gamma}{3-s}.	
	\]
	Such a $j$ exists whenever $\frac{3+\gamma}{2}  < \frac{2-\gamma}{3-s}$, which is equivalent to $\frac{5+5\gamma}{3+\gamma} < s$.
\end{proof}
\begin{rem}
	This proof can be done essentially without change for the power series expansion in $\mcal{F}H^1$ as well.
\end{rem}

\subsection{Breakdown away from the origin}

We now move to part (2) of Theorem \ref{thm:main2}.
We adopt the following abuse of notation: $w_1$ refers both to the function in $X([0,\infty))$ defined by the results of Section \ref{sec:analyticity}, and also to the map
\[
	v\mapsto w_1(v) = v - S\mcal{N}(u,u,w_1).	
\]
We will use a similar convention for $w_1^+$.

To keep things concrete, let us work specifically with $\mcal{T} = \mcal{S}$; the discussion adapts easily to $\mcal{W}$.
We recall the notation
\[
	\mcal{S}(u_0+v) = \mcal{S}(u_0) + \sum_{k=1}^\infty w_k^+.	
\]
Our goal here is the following:
\begin{prop}\label{prop:breakdown_nonzero}
	There exists $R = R(d,\gamma)>0$ such that for all $u_0\in\Sigma$ satisfying $\|u_0\|_\Sigma < R$, all $\|v\|_\Sigma$ small, and all $s > \frac{4+4\gamma}{2+\gamma}$, we have
	\[
		\|\mcal{S}(u_0 + v) - \mcal{S}(u_0) - w_1^+ - w_2^+ \|_2 \neq \mcal{O}(\|v\|_2^s).
	\]
\end{prop}
This will emerge as a consequence of the following estimate on the third derivative term:
\begin{prop}\label{prop:w_3_estimate}
	There exists $R = R(d,\gamma) > 0$ such that for all  $\|u_0\|_\Sigma < R$ and for $\ve \ll 1$, $\ve\sigma \ll 1$, $\sigma \gg 1$, we have
	\[
		\|w_3^+(v_{\ve,\sigma})\|_2 \gtrsim_{d,\gamma} \ve^3\sigma^{2-\gamma}.
	\]
\end{prop}
The crux of the proof is to identify the source of the breakdown of regularity in $w_3^+$, which has the form
\[
	w_3^+ = S\mcal{N}(u,u,w_3) + S\mcal{N}(u,w_1,w_2) + \mcal{N}(w_1,w_1,w_1).	
\]
The bad behavior we seek arises from the term $\mcal{N}(w_1,w_1,w_1)$, which we expect to be dominant as it is essentially a resonant interaction.
We will first establish that this resonant term has the optimal scaling $\ve^3\sigma^{2-\gamma}$ as was the case at $u_0 = 0$, and then show that the remaining cubic terms are subdominant.

\subsubsection{The main term}

We continue to work in the regime $\ve\ll 1,\sigma\gg 1,\ve\sigma \ll 1$.
Our goal is to establish the following:
\begin{prop}\label{prop:resonant_estimate}
	There exists $R = R(d,\gamma)>0$ such that if $\|u_0\|_\Sigma < R$, then
	\[
		\|\mcal{N}(w_1(v_{\ve,\sigma}),w_1(v_{\ve,\sigma}),w_1(v_{\ve,\sigma}))\|_2 \gtrsim \ve^3\sigma^{2-\gamma}.
	\]
	for all $\sigma$ sufficiently large.
\end{prop}
When $u_0\neq 0$, the fact that $w_1$ is no longer purely a free evolution complicates the proof of this relationship; we cannot use the scaling argument for the potential energy
\[
	\int_0^\infty Q(e^{is\Delta}v)~ds	
\]
right away.
We get around this issue by using the fact that $w_1$ \ita{does} behave like a free evolution at large times, and rescaling in time so that the rescaled version of $w_1$ behaves like a free evolution at almost all times.

\begin{lem}\label{lem:shared_scaling}
	There exists $R = R(d,\gamma)>0$ such that if $\|u_0\|_\Sigma < R$, then
	\begin{align*}
		\|w_1\|_{Y([0,\infty))} 
			&\sim \|w_1^+\|_2 
			\sim \|v\|_2,\\
		\|w_1\|_{X([0,\infty))} 
			&\sim \|w_1^+\|_\Sigma 
			\sim \|v\|_\Sigma,\\
		\|\nabla w_1\|_{\infty,2}
			&\lesssim \|v\|_{H^1}.
	\end{align*}
\end{lem}
\begin{proof}
	Writing $w_1^+ = v + S\mcal{N}(u,u,w_1)$ and arguing as usual,
	\[
		\|w_1^+\|_2 \leq \|w_1\|_{Y([0,\infty))} \lesssim \|v\|_2 + \|u\|_{X([0,\infty))}^2\|w_1\|_{Y([0,\infty))}.
	\]
	Taking $\|u_0\|_\Sigma$ small, we can make $\|u\|_{X([0,\infty))}$ arbitrarily small.
	Doing so, we find that
	\[
		\|w_1^+\|_2 \leq \|w_1\|_{Y([0,\infty))} \lesssim \|v\|_2.
	\]
	Similarly, we have
	\[
		\|v\|_2 \leq \|w_1^+\|_2 + \|u\|_{X([0,\infty))}^2\|w_1\|_{Y([0,\infty))} \lesssim \|w_1^+\|_2 + \|u\|_{X([0,\infty))}^2\|v\|_2,
	\]
	and thus taking $\|u_0\|_\Sigma$ small we obtain
	\[
		\|v\|_2 \lesssim \|w_1^+\|_2.	
	\]
	The other estimates follow similarly.
\end{proof}
In particular, we have the scaling
\[
	\|w_1(v_{\ve,\sigma})\|_{Y([0,\infty))} \sim \|w_1^+(v_{\ve,\sigma})\|_2 \sim \|(w_1^+)_{\ve,\sigma}\|_2 \sim \ve,
\]
\[
	\|w_1(v_{\ve,\sigma})\|_{X([0,\infty))} \sim \|w_1^+(v_{\ve,\sigma})\|_\Sigma \sim \|x(w_1^+)_{\ve,\sigma}\|_2\sim \ve\sigma;
\]
this will be useful because most quantities we estimate henceforth depend more directly on $w_1$ and $w_1^+$ than on $v$.

\begin{proof}[Proof of Proposition \ref{prop:resonant_estimate}]
	Fix $v\neq 0$.
	For any $\tau>0$, $L^2$-duality, Fubini, and unitarity of the free propagator we have
	\begin{align*}
		\|\mcal{N}(w_1,w_1,w_1)\|_2
			&\geq \frac{1}{\|w_1^+\|_2} \left| \int_\tau^\infty \langle T(w_1,w_1,w_1), e^{is\Delta}w_1^+ \rangle_{L_x^2}~ds \right| \\
			&\hspace{1em}- \left\| \int_0^\tau e^{-is\Delta} T(w_1,w_1,w_1)~ds\right\|_2.
	\end{align*}
	We write
	\begin{align*}
		\int_\tau^\infty \langle T(w_1,w_1,w_1), e^{is\Delta}w_1^+ \rangle_{L_x^2}~ds
			&= \int_\tau^\infty Q(e^{is\Delta}w_1^+)~ds\\
			&- \int_\tau^\infty \langle ST(w_1,w_1,e^{is\Delta}w_1^+ - w_1), e^{is\Delta}w_1^+\rangle_{L_x^2}~ds.
	\end{align*}
	Since $\|e^{is\Delta}w_1^+ - w_1\|_{Y([\tau,\infty))} \to 0$ as $\tau\to\infty$, it follows that for sufficiently large $\tau = \tau(v)$ we have
	\[
		\int_\tau^\infty \langle T(w_1,w_1,w_1), e^{is\Delta}w_1^+ \rangle_{L_x^2}~ds
		\sim \int_\tau^\infty Q(e^{is\Delta}w_1^+)~ds
	\]
	We now rescale $w_1^+ \mapsto (w_1^+)_{\ve,\sigma}$, which (at least on $L^2$ and $\Sigma$) is essentially equivalent to rescaling $v\mapsto v_{\ve,\sigma}$.
	Under this rescaling, the parabolic scaling symmetry of the free Schr\"odinger flow yields
	\[
		\int_\tau^\infty Q(e^{is\Delta}(w_1^+)_{\ve,\sigma})~ds
			= \ve^4\sigma^{2-\gamma}\int_{\tau/\sigma^2}^\infty Q(e^{is\Delta}w_1^+)~ds.
	\]
	Thinking of $\sigma$ being arbitrarily large, we estimate this by the integral over all of $[0,\infty)$.
	Using Lemma \ref{lem:multilinear_estimate} and the Gagliardo-Nirenberg inequality, this incurs an error of size
	\[
		\int_0^{\tau/\sigma^2} Q(e^{is\Delta}w_1^+)~ds 
			\lesssim \frac{\tau}{\sigma^2}(\|w_1^+\|_2^{1-\theta}\|\nabla w_1^+\|_2^\theta)^4
			\lesssim \frac{\tau}{\sigma^2}\|v\|_\Sigma^4 \ll 1.
	\]
	This establishes the $\ve^3\sigma^{2-\gamma}$ scaling on the main term.
	By Strichartz, Gagliardo-Nirenberg, and Lemma \ref{lem:shared_scaling} we find that the remainder satisfies
	\begin{align*}
		\left\|\int_0^\tau e^{-is\Delta} T(w_1,w_1,w_1)~ds\right\|_2	
			&\lesssim \|w_1\|_{L_t^\alpha L_x^r([0,T])}^2 \|w_1\|_{Y([0,T])}\\
			& \lesssim \left( \int_0^\tau (\|w_1\|_{\infty,2}^{1-\theta}\|\nabla w_1\|_{\infty,2}^\theta)^\alpha \right)^{2/\alpha}\|v\|_2\\
			&\lesssim \tau^{2/\alpha} \|v\|_{H^1}^3.
	\end{align*}
	Since we are working in the regime $\sigma\gg 1$, $\|v_{\ve,\sigma}\|_{H^1} \sim \ve$.
	Therefore rescaling yields
	\[
		\left\|\int_0^\tau e^{-is\Delta} T(w_1(v_{\ve,\sigma}),w_1(v_{\ve,\sigma}),w_1(v_{\ve,\sigma}))~ds\right\|_2
			\lesssim \tau^{2/\alpha}\ve^3.
	\]
	Since $\tau$ depends only on $v$ and $2-\gamma > 0$, this term is subdominant to $\ve^3\sigma^{2-\gamma}$ for $\sigma\gg 1$.
\end{proof}

\subsubsection{Nonresonant cubic terms}
We are left to control the nonresonant cubic error terms $\|S\mcal{N}(u,w_1,w_2)\|_2$ and $\|S\mcal{N}(u,u,w_3)\|_2$.
\begin{prop}\label{prop:nonresonant_cubic_errors}
	There exists $R = R(d,\gamma) > 0$ such that for all  $\|u_0\|_\Sigma < R$ and for $\ve \ll 1$, $\ve\sigma \ll 1$, $\sigma \gg 1$ we have
	\[
		\|S\mcal{N}(u,w_1,w_2)\|_2 + \|S\mcal{N}(u,u,w_3)\|_2
			\lesssim R^2\ve^3\sigma^{2-\gamma}.
	\]
\end{prop}
Taking $R$ small will ensure that the main term continues to dominate the errors.
Together with Proposition \ref{prop:resonant_estimate} this immediately implies Proposition \ref{prop:w_3_estimate}.

The idea of this proof is that the estimate $\|w_1\|_{\alpha,r} \lesssim \|v\|_\Sigma$, which holds due to the fact that $|t|^{-\frac{\alpha\gamma}{4}}$ is integrable near $\infty$, is slack.
Na\"ively using it to control $\|S\mcal{N}(u,w_1,w_2)\|_2$ and $\|S\mcal{N}(u,u,w_3)\|_2$ only yields an estimate of $\ve^3\sigma^3$, which is \ita{not} subdominant to $\ve^3\sigma^{2-\gamma}$.
For small data, the estimate can be improved to one of the form $\|w_1\|_{\alpha,r} \lesssim \|v\|_\Sigma^a\|v\|_2^b\|\nabla v\|_2^c$ by relying less on Lemma \ref{lem:commutator_decay}.
This idea was also used in \cites{Lee21} to sharpen an analogous error estimate in order to recover dimensions $d=1,2,3$ in Theorem \ref{thm:pNLS_main}.
The Lorentz space refinement is used to recover an endpoint, which is necessary as the scaling $\ve^3\sigma^{2-\gamma}$ is sharp.
\begin{lem}\label{lem:Lorentz_optimized_estimate}
	There exists $R = R(d,\gamma)>0$ such that if $\|u_0\|_\Sigma < R$, then
	\begin{align*}
		\|w_1\|_{L_t^{\alpha,\infty}L_x^r([0,\infty))}
			\lesssim \|v\|_2^{1-\gamma/4}\|v\|_\Sigma^{(4-\gamma)/8}\|\nabla w_1\|_{\infty,2}^{(3\gamma-4)/8}.
	\end{align*}
\end{lem}
\begin{proof}
	For any $A\in[0,1]$, by Lemma \ref{lem:commutator_decay} and the Gagliardo-Nirenberg inequality we have
	\[
		\|w_1(t)\|_r \lesssim |t|^{-A\gamma/4}\|w_1(t)\|_2^{1-\gamma/4}\|J(t)w_1(t)\|_2^{A\gamma/4}\|\nabla w_1(t)\|_2^{(1-A)\gamma/4}.
	\]
	Choose $A = \frac{4}{\alpha\gamma} = \frac{4-\gamma}{2\gamma}$.
	Note that for $\frac{4}{3}<\gamma<2$, $\frac{1}{2} < \frac{4-\gamma}{2\gamma}<1$.
	Since $|t|^{-A\gamma/4}\in L_t^{\alpha,\infty}([0,\infty))$ for this choice of $A$, we find that
	\begin{align*}
		\|w_1\|_{L_t^{\alpha,\infty}L_x^r([0,\infty))}
			\lesssim \|w_1(t)\|_{\infty,2}^{1-\gamma/4}\|J(t)w_1(t)\|_{\infty,2}^{(4-\gamma)/8}\|\nabla w_1(t)\|_{\infty,2}^{(3\gamma-4)/8}.
	\end{align*}
	The claim now follows from the Lorentz space version of Lemma \ref{lem:shared_scaling}; we leave the details to the reader.
\end{proof}
\begin{lem}\label{lem:nabla_bootstrap}
	There exist $R = R(d,\gamma) > 0$ and $L = L(d,\gamma) > 0$ such that for all  $\|u_0\|_\Sigma < R$ and $v\in\Sigma$ with $\|v\|_\Sigma \leq L$, we have
	\[
		\|\nabla w_1^+\|_2 + \|\nabla w_1\|_{Y([0,\infty))} \lesssim \|\nabla v\|_2.	
	\]
\end{lem}
\begin{proof}
	By the usual estimates and Lemma \ref{lem:Lorentz_optimized_estimate}, for all $\tau>0$ we have
	\begin{align*}
		\|\nabla w_1\|_{Y([0,\tau))}
			&\lesssim_{d,\gamma} \|\nabla v\|_2 + \|u\|_{X([0,\tau))}^2(\|w_1\|_{L_t^{\alpha,\infty}L_x^r([0,\tau))} + \|\nabla w_1\|_{Y([0,\tau))})\\
			&\lesssim_{d,\gamma} \|\nabla v\|_2 + R^2(\|v\|_2^{1-\gamma/4}\|v\|_\Sigma^{(4-\gamma)/8}\|\nabla w_1\|_{Y([0,\tau))}^{(3\gamma-4)/8} + \|\nabla w_1\|_{Y([0,\tau))})\\
			&\leq \|\nabla v\|_2 + \delta_1\|\nabla w_1\|_{Y([0,\tau))}^{(3\gamma-4)/8} + \delta_2\|\nabla w_1\|_{Y([0,\tau))},
	\end{align*}
	where $\delta_1 = \delta_1(R,L)$ and $\delta_2 = \delta_2(R)$ can be made arbitrarily small by an appropriate choice of $R$ and $L$.
	The estimate for $\nabla w_1$ then follows via a bootstrap argument.
	The estimate for $\nabla w_1^+$ is then immediate from the definition of $w_1^+$.
\end{proof}
\begin{proof}[Proof of Proposition \ref{prop:nonresonant_cubic_errors}]
	By the Lorentz space adaptations of our earlier estimates, we have
	\begin{align*}
		\|S\mcal{N}(u,w_1,w_2)\|_2
			&\lesssim \|u\|_{X^*(I)}\|w_1\|_{L_t^{\alpha,\infty}L_x^r}\|w_2\|_{L_t^{q,2}L_x^r}\\
			&\lesssim \|u\|_{X^*(I)}^2\|w_1\|_{L_t^{\alpha,\infty}L_x^r}^2\|w_1\|_{L_t^{q,2}L_x^r}
	\end{align*}
	and
	\begin{align*}
		\|S\mcal{N}(u,u,w_3)\|_2
			&\lesssim \|u\|_{X^*(I)}^2\|w_3\|_{L_t^{q,2}L_x^r}\\
			&\lesssim \|u\|_{X^*(I)}^2(\|u\|_{X^*(I)}\|w_1\|_{L_t^{\alpha,\infty}L_x^r}\|w_2\|_{L_t^{q,2}L_x^r} + \|w_1\|_{L_t^{\alpha,\infty}L_x^r}^2\|w_1\|_{L_t^{q,2}L_x^r})\\
			&\lesssim \|u\|_{X^*(I)}^2(\|u\|_{X^*(I)}^2\|w_1\|_{L_t^{\alpha,\infty}L_x^r}^2\|w_1\|_{L_t^{q,2}L_x^r} + \|w_1\|_{L_t^{\alpha,\infty}L_x^r}^2\|w_1\|_{L_t^{q,2}L_x^r})\\
			&= \|u\|_{X^*(I)}^2(1 + \|u\|_{X^*(I)}^2)\|w_1\|_{L_t^{\alpha,\infty}L_x^r}^2\|w_1\|_{L_t^{q,2}L_x^r}.
	\end{align*}
	The claim now follows by taking $R = R(d,\gamma)$ small, $\|u_0\|_\Sigma<R$, replacing $w_1$ with $w_1(v_{\ve,\sigma})$, and applying Lemmas \ref{lem:shared_scaling}, \ref{lem:Lorentz_optimized_estimate}, and \ref{lem:nabla_bootstrap}.
\end{proof}

\subsubsection{Conclusion of the proof}

At last, we are ready to proceed with the proof of Proposition \ref{prop:breakdown_nonzero}, thereby completing the proof of Theorem \ref{thm:main2}.

\begin{proof}
	We define $v_{\ve,\sigma}$ as before and work in the regime $\ve\ll 1,\sigma\gg 1,\ve\sigma \ll 1$.
	We write
	\[
		\mcal{S}(u_0+v) - \mcal{S}(u_0) - w_1^+ - w_2^+ = w_3^+(v) + e(v),	
	\]
	where
	\[
		e(v) = \sum_{k\geq 4} w_k^+.	
	\]
	By Propositions \ref{prop:coefficient_estimates} and \ref{prop:w_3_estimate}, for sufficiently small $\|u_0\|_\Sigma$ we have the lower bound
	\begin{align*}
		\|\mcal{S}(u_0+v) - \mcal{S}(u_0) - w_1^+ - w_2^+\|_2
			&\geq \|w_3^+(v_{\ve,\sigma})\|_2 - \|e(v_{\ve,\sigma}\|_2\\
			&\gtrsim \ve^3\sigma^{2-\gamma} - \ve^4\sigma^4.
	\end{align*}
	We take $\ve = \sigma^{-j}$ with $j>1$ to be determined: then $\ve\sigma\ll 1$ for $\sigma\gg 1$.
	For the main term to dominate the quartic error we require $-3j + 2 - \gamma > -4j + 4$, which is equivalent to $j > 2+\gamma$ and supersedes the condition $j>1$.
	Assuming this condition, we have
	\[
		\frac{1}{\|v_{\ve,\sigma}\|^2}\|\mcal{S}(u_0+v_{\ve,\sigma}) - \mcal{S}(u_0) - w_1^+ - w_2^+\|_2	
			\gtrsim \sigma^{(s-3)j + 2 - \gamma}.
	\]
	The RHS is unbounded as $\sigma\to\infty$ as long as $(s-3)j + 2 - \gamma > 0$.
	This condition is automatically met if $s\geq3$ since $\gamma < 2$, while if $s< 3$ then it is equivalent to $j < \frac{2-\gamma}{3-s}$.
	Thus an appropriate value of $j$ can be found provided that $2+\gamma < \frac{2-\gamma}{3-s}$, which is equivalent to the condition $s > \frac{4+4\gamma}{2+\gamma}$.
\end{proof}


\section*{Acknowledgments} The author thanks his advisors Rowan Killip and Monica Vi\c san for many helpful discussions and guidance.
The author also thanks Kenji Nakanishi for suggesting the study of the Hartree NLS as a follow-up to \cites{Lee21}.
This work was supported by NSF grants 1600942 (principal investigator: Rowan Killip) and 1500707 (principal investigator: Monica Vi\c{s}an).

\bibliography{main}

\begin{bibdiv}
\begin{biblist}

\bib{Andersson97}{article}{
      author={Andersson, P.},
       title={Characterization of pointwise {H}\"{o}lder regularity},
        date={1997},
        ISSN={1063-5203},
     journal={Appl. Comput. Harmon. Anal.},
      volume={4},
      number={4},
       pages={429\ndash 443},
         url={https://doi.org/10.1006/acha.1997.0219},
      review={\MR{1474098}},
}

\bib{BeTa06}{article}{
      author={Bejenaru, I.},
      author={Tao, T.},
       title={Sharp well-posedness and ill-posedness results for a quadratic
  non-linear {S}chr\"{o}dinger equation},
        date={2006},
        ISSN={0022-1236},
     journal={J. Funct. Anal.},
      volume={233},
      number={1},
       pages={228\ndash 259},
         url={https://doi.org/10.1016/j.jfa.2005.08.004},
      review={\MR{2204680}},
}

\bib{Bo97}{article}{
      author={Bourgain, J.},
       title={Periodic {K}orteweg de {V}ries equation with measures as initial
  data},
        date={1997},
        ISSN={1022-1824},
     journal={Selecta Math. (N.S.)},
      volume={3},
      number={2},
       pages={115\ndash 159},
         url={https://doi.org/10.1007/s000290050008},
      review={\MR{1466164}},
}

\bib{CaGa09}{article}{
      author={Carles, R.},
      author={Gallagher, I.},
       title={Analyticity of the scattering operator for semilinear dispersive
  equations},
        date={2009},
        ISSN={0010-3616},
     journal={Comm. Math. Phys.},
      volume={286},
      number={3},
       pages={1181\ndash 1209},
         url={https://doi.org/10.1007/s00220-008-0599-x},
      review={\MR{2472030}},
}

\bib{CaOz08}{article}{
      author={Carles, R.},
      author={Ozawa, T.},
       title={On the wave operators for the critical nonlinear
  {S}chr\"{o}dinger equation},
        date={2008},
        ISSN={1073-2780},
     journal={Math. Res. Lett.},
      volume={15},
      number={1},
       pages={185\ndash 195},
         url={https://doi.org/10.4310/MRL.2008.v15.n1.a15},
      review={\MR{2367183}},
}

\bib{CaWe92}{article}{
      author={Cazenave, T.},
      author={Weissler, F.~B.},
       title={Rapidly decaying solutions of the nonlinear {S}chr\"{o}dinger
  equation},
        date={1992},
        ISSN={0010-3616},
     journal={Comm. Math. Phys.},
      volume={147},
      number={1},
       pages={75\ndash 100},
         url={http://projecteuclid.org/euclid.cmp/1104250527},
      review={\MR{1171761}},
}

\bib{GiVe80}{article}{
      author={Ginibre, J.},
      author={Velo, G.},
       title={On a class of nonlinear {S}chr\"{o}dinger equations with nonlocal
  interaction},
        date={1980},
        ISSN={0025-5874},
     journal={Math. Z.},
      volume={170},
      number={2},
       pages={109\ndash 136},
         url={https://doi.org/10.1007/BF01214768},
      review={\MR{562582}},
}

\bib{Glassey77}{article}{
      author={Glassey, R.~T.},
       title={Asymptotic behavior of solutions to certain nonlinear
  {S}chr\"{o}dinger-{H}artree equations},
        date={1977},
        ISSN={0010-3616},
     journal={Comm. Math. Phys.},
      volume={53},
      number={1},
       pages={9\ndash 18},
         url={http://projecteuclid.org/euclid.cmp/1103900588},
      review={\MR{486956}},
}

\bib{HaOz88}{article}{
      author={Hayashi, N.},
      author={Ozawa, T.},
       title={Scattering theory in the weighted {$L^2({\bf R}^n)$} spaces for
  some {S}chr\"{o}dinger equations},
        date={1988},
        ISSN={0246-0211},
     journal={Ann. Inst. H. Poincar\'{e} Phys. Th\'{e}or.},
      volume={48},
      number={1},
       pages={17\ndash 37},
         url={http://www.numdam.org/item?id=AIHPB_1988__48_1_17_0},
      review={\MR{947158}},
}

\bib{HaTs87}{article}{
      author={Hayashi, N.},
      author={Tsutsumi, Y.},
       title={Scattering theory for {H}artree type equations},
        date={1987},
        ISSN={0246-0211},
     journal={Ann. Inst. H. Poincar\'{e} Phys. Th\'{e}or.},
      volume={46},
      number={2},
       pages={187\ndash 213},
         url={http://www.numdam.org/item?id=AIHPB_1987__46_2_187_0},
      review={\MR{887147}},
}

\bib{HoTz10}{article}{
      author={Holmer, J.},
      author={Tzirakis, N.},
       title={Asymptotically linear solutions in {$H^1$} of the 2-{D}
  defocusing nonlinear {S}chr\"{o}dinger and {H}artree equations},
        date={2010},
        ISSN={0219-8916},
     journal={J. Hyperbolic Differ. Equ.},
      volume={7},
      number={1},
       pages={117\ndash 138},
         url={https://doi.org/10.1142/S0219891610002049},
      review={\MR{2646800}},
}

\bib{Kishimoto18}{article}{
      author={Kishimoto, N.},
       title={A remark on norm inflation for nonlinear {S}chr\"{o}dinger
  equations},
        date={2019},
        ISSN={1534-0392},
     journal={Commun. Pure Appl. Anal.},
      volume={18},
      number={3},
       pages={1375\ndash 1402},
         url={https://doi.org/10.3934/cpaa.2019067},
      review={\MR{3917712}},
}

\bib{Kita03}{article}{
      author={Kita, N.},
       title={Sharp {$L^r$} asymptotics of the small solutions to the nonlinear
  {S}chr\"{o}dinger equations},
        date={2003},
        ISSN={0362-546X},
     journal={Nonlinear Anal.},
      volume={52},
      number={4},
       pages={1365\ndash 1377},
         url={https://doi.org/10.1016/S0362-546X(02)00265-1},
      review={\MR{1941262}},
}

\bib{KiOz05}{article}{
      author={Kita, N.},
      author={Ozawa, T.},
       title={Sharp asymptotic behavior of solutions to nonlinear
  {S}chr\"{o}dinger equations with repulsive interactions},
        date={2005},
        ISSN={0219-1997},
     journal={Commun. Contemp. Math.},
      volume={7},
      number={2},
       pages={167\ndash 176},
         url={https://doi.org/10.1142/S0219199705001696},
      review={\MR{2140548}},
}

\bib{Lee21}{article}{
      author={Lee, G.~E.},
       title={Breakdown of regularity of scattering for mass-subcritical
  {NLS}},
        date={2021},
        ISSN={1073-7928},
     journal={Int. Math. Res. Not. IMRN},
      number={5},
       pages={3571\ndash 3596},
         url={https://doi.org/10.1093/imrn/rnaa072},
      review={\MR{4227579}},
}

\bib{Masaki09}{article}{
      author={Masaki, S.},
       title={Asymptotic expansion of solutions to the nonlinear
  {S}chr\"{o}dinger equation with power nonlinearity},
        date={2009},
        ISSN={1340-6116},
     journal={Kyushu J. Math.},
      volume={63},
      number={1},
       pages={51\ndash 82},
         url={https://doi.org/10.2206/kyushujm.63.51},
      review={\MR{2522922}},
}

\bib{Masaki19}{inproceedings}{
      author={Masaki, S.},
       title={On the scattering problem of mass-subcritical {H}artree
  equation},
        date={2019},
   booktitle={Asymptotic analysis for nonlinear dispersive and wave equations},
   publisher={Mathematical Society of Japan},
     address={Tokyo, Japan},
       pages={259\ndash 309},
         url={https://doi.org/10.2969/aspm/08110259},
}

\bib{MiWuZh09}{article}{
      author={Miao, C.},
      author={Wu, H.},
      author={Zhang, H.},
       title={On the real analyticity of the scattering operator for the
  {H}artree equation},
        date={2009},
        ISSN={0066-2216},
     journal={Ann. Polon. Math.},
      volume={95},
      number={3},
       pages={227\ndash 242},
         url={https://doi.org/10.4064/ap95-3-3},
      review={\MR{2491379}},
}

\bib{NaOz92}{article}{
      author={Nawa, H.},
      author={Ozawa, T.},
       title={Nonlinear scattering with nonlocal interaction},
        date={1992},
        ISSN={0010-3616},
     journal={Comm. Math. Phys.},
      volume={146},
      number={2},
       pages={259\ndash 275},
         url={http://projecteuclid.org/euclid.cmp/1104250192},
      review={\MR{1165183}},
}

\bib{Strauss81}{article}{
      author={Strauss, W.~A.},
       title={Nonlinear scattering theory at low energy},
        date={1981},
        ISSN={0022-1236},
     journal={J. Functional Analysis},
      volume={41},
      number={1},
       pages={110\ndash 133},
         url={https://doi.org/10.1016/0022-1236(81)90063-X},
      review={\MR{614228}},
}

\bib{Strauss81.sequel}{article}{
      author={Strauss, W.~A.},
       title={Nonlinear scattering theory at low energy: sequel},
        date={1981},
        ISSN={0022-1236},
     journal={J. Functional Analysis},
      volume={43},
      number={3},
       pages={281\ndash 293},
         url={https://doi.org/10.1016/0022-1236(81)90019-7},
      review={\MR{636702}},
}

\end{biblist}
\end{bibdiv}

\end{document}